\documentclass[11pt]{amsart}
\usepackage{graphicx}
\usepackage[USenglish,english]{babel}
\usepackage[T1]{fontenc}
\usepackage[latin1]{inputenc}
\usepackage{amsfonts}
\usepackage{amssymb}
\usepackage{amsthm}
\usepackage{amsmath}
\usepackage{fancyhdr}

\usepackage{tikz}
\usetikzlibrary{decorations.pathreplacing}
\usetikzlibrary{shapes.misc}
\tikzset{cross/.style={cross out, draw=black, fill=none, minimum size=2*(#1-\pgflinewidth), inner sep=0pt, outer sep=0pt}, cross/.default={2pt}}
\usepackage{float}
\usepackage{enumerate}
\usepackage{accents}
\usepackage{mathtools}
\usepackage{mathrsfs}
\usepackage{comment}
\usepackage{afterpage}
\usepackage{bbm}
\usepackage{dsfont}
\usepackage{stmaryrd}
\usepackage{mleftright}
\usepackage{subfig}
\usepackage{faktor}
\usepackage{graphicx}
\usepackage{marvosym}
\usepackage[all]{xy}
\usepackage[linktocpage]{hyperref}
\usepackage{xcolor}

\newtheorem{theorem}{Theorem}
\numberwithin{theorem}{section}

\newtheorem{lemma}[theorem]{Lemma}

\newtheorem{proposition}[theorem]{Proposition}

\newtheorem{corollary}[theorem]{Corollary}

\newtheorem{thmintro}{Theorem}

\theoremstyle{remark}
\newtheorem{rmk}[theorem]{Remark}

\theoremstyle{definition}
\newtheorem{definition}[theorem]{Definition}

\newcommand{\R}{\mathbb{R}}

\newcommand{\Z}{\mathbb{Z}}
\newcommand{\Q}{\mathbb{Q}}

\newcommand{\Hy}{\mathbb{H}}

\newcommand{\what}{\widehat}
\newcommand{\til}{\widetilde}

\newcommand{\Mtil}{\what{M}}

\newcommand{\dinf}{\partial_\infty}

\DeclareMathOperator{\Isom}{Isom}

\title[The Normal growth exponent of a Codimension-1 hypersurface]{The normal growth exponent of a codimension-1 hypersurface of a negatively curved manifold }

\author{Corey Bregman}
\address{University of Southern Maine, Portland, ME USA}
\email{corey.bregman@maine.edu}

\author{Merlin Incerti-Medici}
\address{IHES, France}
\email{merlin.medici@gmail.com}

\begin{document}
\maketitle

\begin{abstract}
Let $X$ be a Hadamard manifold  with pinched negative curvature $-b^2\leq\kappa\leq -1$. Suppose $\Sigma\subseteq X$ is a totally geodesic, codimension-1 submanifold and consider the geodesic flow $\Phi^\nu_t$ on $X$ generated by a unit normal vector field $\nu$ on $\Sigma$.  
We say the normal growth exponent of $\Sigma$ in $X$ is at most $\beta$ if
\[ \lim_{t \rightarrow \pm \infty} \frac{ \Vert d \Phi_t^\nu \Vert_{\infty} }{ e^{\beta \vert t \vert}} < \infty, \]
where $\Vert d \Phi_t^\nu \Vert_{\infty} $ is the supremum of the operator norm of $d \Phi_t^\nu $ over all points of $\Sigma$.
We show that if $\Sigma$ is bi-Lipschitz to hyperbolic $n$-space $\Hy^n$ and the normal growth exponent is at most 1, then $X$ is bi-Lipschitz to $\Hy^{n+1}$. As an application, we prove that if $M$ is a closed, negatively curved $(n+1)$-manifold, and $N\subset M$ is a totally geodesic, codimension-1 submanifold that is bi-Lipschitz to a hyperbolic manifold and whose normal growth exponent is at most 1, then $\pi_1(M)$ is isomorphic to a lattice in $\Isom(\Hy^{n+1})$. Finally, we show that the assumption on the normal growth exponent is necessary in dimensions at least 4.  
\end{abstract}

\tableofcontents

\section{Introduction}
Let $X$ be a Hadamard manifold with pinched negative curvature $-b^2\leq \kappa\leq -1$. In this paper, we consider a totally geodesic, codimension-1 hypersurface $\Sigma\subseteq X$, and study the dynamical properties of the geodesic flow on $X$ generated by a unit normal vector field to $\Sigma$. 

\subsection{The normal growth exponent}Let $\nu$ be a unit normal vector to $\Sigma$, and let $\Phi^\nu_t$ be the geodesic flow on $X$ generated by $\nu$.  We define the \emph{normal growth exponent} of $\Sigma$ in $X$ to be at most $\beta$ if
\[ \lim_{t \rightarrow \pm \infty} \frac{ \Vert d \Phi_t^\nu \Vert_{\infty} }{ e^{\beta \vert t \vert}} < \infty, \]
where $\Vert d \Phi_t^\nu \Vert_{\infty} $ is the supremum of the operator norm of $d \Phi_t^\nu $ over all points of $\Sigma$. The normal growth exponent measures the distortion on a normal push-out of $\Sigma$ as one moves farther and farther away. If $X$ is real hyperbolic $n$-space $\Hy^{n}$, an easy calculation shows that the normal growth exponent of a totally geodesic copy of $\Hy^{n}$ is 1. Much of our motivation for this work came from a desire to understand this example.  We obtain the following characterisation:

\begin{thmintro}\label{thm:MainNoPi1}
Let $X$ be a Hadamard manifold with pinched negative curvature in $-b^2\leq \kappa\leq -1$ and suppose $\Sigma\subseteq X$ is a totally geodesic, codimension-1 submanifold. If the normal growth exponent of $\Sigma$ in $X$ is at most 1 and $\Sigma$ is bi-Lipschitz to $\Hy^n$, then $X$ is bi-Lipschitz to $\Hy^{n+1}$.
\end{thmintro}

 Note that when $\Sigma$ admits a cocompact action by isometries, the normal growth exponent is at most $b$ (cf. Lemmas \ref{lem:CurvatureBoundyieldsNormalGrowthBound} and \ref{lem:CurvatureBoundyieldsLowerBound}).  Now let $M$ be a closed, negatively curved Riemannian manifold and suppose $N\looparrowright M$ is an immersed, totally geodesic submanifold.  The immersion lifts to a convex embedding of universal covers $\widetilde{N}\hookrightarrow \widetilde{M}$. In this setting, we define the normal growth exponent of $N$ in $M$ to be the normal growth exponent of $\widetilde{N}$ in $\widetilde{M}$. As a consequence of Theorem \ref{thm:MainNoPi1}, we thus also get

\begin{thmintro}\label{thm:Main Cocompact}
Let $M$ be a closed, negatively curved Riemannian manifold.  Suppose that $N\subseteq M$ is a totally geodesic, codimension-1 submanifold whose normal growth exponent is at most 1.  If $N$ is bi-Lipschitz to a manifold of constant negative curvature, then $M$ is homotopy equivalent to a manifold of constant negative curvature.
\end{thmintro}

\begin{rmk}
Let $n$ be the dimension of $M$.  If $n\neq 4$, then $M$ is actually homeomorphic to a manifold of constant negative curvature. This follows from the classification of surfaces in dimension 2, a theorem of Gabai--Meyerhoff--Thurston in dimension 3 \cite{GabaiMeyerhoffThurston03}, and Farrell--Jones' solution to the Borel conjecture for negatively curved manifolds when $n\geq 5$ \cite{FarrellJones89-1,FarrellJones89-2}. 
\end{rmk}

\subsection{Gromov--Thurston Manifolds} One might wonder whether the assumption on the normal growth exponent is necessary in the two theorems above.  Indeed, if $M$ has dimension $\leq 3$ and admits a negatively curved metric then $M$ admits a metric of constant curvature $-1$. This follows from classification of surfaces in dimension 2 and from Perelman's celebrated solution of the Geometrisation conjecture in dimension 3 \cite{Perelman:2002-1,Perelman:2003-1, Perelman:2003-2}.  

In each dimension $n\geq 4$, Gromov and Thurston constructed examples of closed $n$-manifolds which admit metrics of negative curvature but no constant curvature-$(-1)$ metric \cite{GromovThurston87}. In fact, their argument shows that in certain cases, such examples are not even homotopy equivalent a metric with constant negative curvature (see \cite{Kapovich07} for a similar construction).  By results of Sullivan \cite{Sullivan81} and Tukia \cite{Tukia86}, an aspherical manifold $M$ is homotopy equivalent a hyperbolic manifold if and only if the $\pi_1(M)$ is quasi-isometric to $\Hy^n$. Thus, the fundamental groups of Gromov--Thurston manifolds are hyperbolic but not quasi-isometric to $\Hy^n$.

Gromov--Thurston manifolds are constructed by starting with a hyperbolic manifold $M$ and taking a branched cover of along a totally geodesic  codimension-2 submanifold $V$. The singular Riemannian pulled back from $M$ can be smoothed to yield a negatively curved Riemannian manifold.  However, because the branched cover can be of arbitrary degree, Wang's finiteness theorem implies that only finitely many of these manifolds can admit a metric of constant curvature-$(-1)$ \cite{Wang72}. 

For any $\epsilon>0$, Gromov and Thurston construct examples such that the curvature of the smoothed metric lies in $[-1-\epsilon,-1]$.  Their examples also contain many totally geodesic codimension-1 submanifolds whose induced metrics have constant curvature-$(-1)$ (cf. Proposition \ref{prop:W is totally geodesic}).  Since the normal growth exponent lies in the interval $[1,b]$ in the cocompact case, we have

\begin{thmintro} \label{thm:MainGromovThurstonExamples}
For each $n\geq 4$ and any sequence $\epsilon_i\rightarrow 0$ of positive real numbers, there exists a sequence of pairs $(M_i,N_i)$ where $M_i$ is closed $n$-manifold with curvature in the interval $[-1-\epsilon_i,-1]$ and $N_i\subset M_i$ is a totally geodesic, codimension-1 submanifold satisfying
\begin{itemize}
    \item The induced metric on $N$ has constant curvature $\kappa\equiv -1$
    \item The normal growth exponent of $N_i$ in $M_i$ is at most $1+\epsilon_i$.
    \item $\pi_1(M_i)$ is not quasi-isometric to $\Hy^n$.    
\end{itemize}

\end{thmintro}
Hence, by Theorem \ref{thm:Main Cocompact}, for these examples the normal growth exponent is never at most 1, but gets arbitrarily close.

\subsection{Cubulated groups} For closed manifolds, the existence of a totally geodesic submanifold in a negatively curved manifold often arises from an arithmetic construction. A coarser notion than that of a totally geodesic codimension-1 hypersurface, and one more suited to applications in geometric group theory, is that of a codimension-1 quasi-convex subgroup. If $G$ is a finitely generated group with generating set $S$ and $H\leq G$ is a subgroup, then $H$ is said to be \emph{quasi-convex} if there exists $R>0$ such that any geodesic in the Cayley graph $\Gamma_{G,S}$ between two vertices in $H$ lies within the $R$-neighborhood of $H$.  When $G$ is word-hyperbolic, whether or not $H$ is quasi-convex does not depend on the generating set $S$, and implies that $H$ is itself a word-hyperbolic group. We say that $H$ has \emph{codimension-1} if the coset graph of $\faktor{G}{H}$ has at least two ends at infinity.  

If $M$ is a closed, negatively curved Riemannian manifold and $N\subseteq M$ a totally geodesic submanifold, then $\pi_1(N)$ injects into $\pi_1(M)$ as a quasi-convex subgroup, and if $N$ has codimension-1, then $\pi_1(N)$ has codimension-1 in $\pi_1(M)$.  

Work of Bergeron--Wise has highlighted the importance of having many quasi-convex, codimension-1 subgroups for producing proper, cocompact actions of a hyperbolic group $G$ on CAT(0) cube complexes \cite{BergeronWise12}.  A group is called \emph{cubulated} if it acts properly and cocompactly on a CAT(0) cube complex by isometries. It is not difficult to show directly that hyperbolic surface groups are cubulated. In dimension 3, deep results of Kahn--Markovic produce many quasi-convex surface subgroups in any hyperbolic 3-manifold \cite{KahnMarkovic15}. Combined with the criterion of Bergeron--Wise, this implies every closed hyperbolic 3-manifold group is cubulated \cite{BergeronHaglundWise11}.  

In higher dimensions, it is not known whether every hyperbolic manifold can be cubulated.  For $n\geq 4$, most constructions of hyperbolic $n$-manifolds are either arithmetic or obtained from arithmetic manifolds by cut and paste constructions along totally geodesic, codimension-1 hypersurfaces \cite{GromovPiatetskiShapiro88, Agol06, BelolipetskyThomson11}.  All such constructions therefore contain many totally geodesic codimension-1 submanifolds, and are cubulated \cite{BergeronHaglundWise11}. 

It would be interesting to find either a coarse or CAT(0) cube complex version of the normal growth exponent, and a corresponding version of Theorem \ref{thm:Main Cocompact}. We remark that the Gromov--Thurston examples are cubulated, hence even in this situation some restriction on the normal growth exponent will still be necessary \cite{Giralt17}.

\textbf{Outline.} In \S \ref{sec:Preliminaries}, we review some necessary background on geodesic flows and quasi-isometric rigidity of real hyperbolic space.  We define the normal growth exponent of a totally geodesic hypersurface $\Sigma$ in \S \ref{sec:Normal Growth Estimates}, and bound the exponent in terms of the curvature when $\Sigma$ admits a cocompact group action. In \S \ref{sec:BoundingDistanceswithNormalGrowthExponent}, we prove Theorems \ref{thm:MainNoPi1} and \ref{thm:Main Cocompact}. Finally, we review the Gromov--Thurston examples and prove Theorem \ref{thm:MainGromovThurstonExamples} in \S \ref{sec:GromovThurstonManifolds}.

\textbf{Acknowledgements.} The authors are grateful to Pierre Pansu for bringing the Gromov--Thurston examples to their attention. The second author thanks Fanny Kassel and Thibault Lefeuvre for several discussions about geodesic flows.  The first author was supported by NSF grant DMS-2052801. The second author has been funded by the SNSF grant 194996.

\section{Preliminaries}\label{sec:Preliminaries}
In this section we introduce some important results and notation that will be used in the sequel.  First, we  recall the definition of totally geodesic submanifold and describe some warped product decompositions of real hyperbolic $n$-space $\Hy^n$. We then review some basic results concerning geodesic flows on negatively curved manifolds.  Finally, we discuss a rigidity results for groups quasi-isometric to $\Hy^n$.

\subsection{Totally geodesic submanifolds} \label{subsec:TotallyGeodesicSubmanifolds}
Suppose $(M^{n+1},g)$ is a closed, orientable negatively curved Riemannian manifold with sectional curvature $-b^2\leq \kappa\leq 1$.  Let $X$ denote the universal cover of $M$, which diffeomorphic to $\R^{n+1}$ by the Cartan--Hadamard theorem.  Suppose that $\iota\colon N^n\hookrightarrow M$ is an embedded, orientable submanifold, and let $h=\iota^*(g)$ be the induced Riemannian metric on $N$.

\begin{definition}
$(N,h)$ is \emph{totally geodesic} if any geodesic on $N$ is also a geodesic on $M$.
\end{definition}

Then $N$ is itself a negatively curved Riemannian manifold. If $Y$ is the universal cover of $N$, then there is a convex embedding $Y\hookrightarrow X$. Hence $Y$ is diffeomorphic to $\R^n$ and inclusion $\iota\colon N\rightarrow M$ is $\pi_1$-injective. Suppose $N$ is a codimension-1 submanifold of $M$, and let $\nu$ be the unit normal vector field on $N$ with positive orientation.  Then $N$ being totally geodesic is equivalent to $\nu$ being parallel along $N$, \emph{i.e.} that $\nabla_w\nu=0$ for any $w\in TN$. 

Let $\iota : Y \hookrightarrow X$ be a $C^2$-embedded, totally geodesic, codimension-1 submanifold and $\nu$ a $C^1$-differentiable unit normal vector field on $Y$. We define the {\it normal flow with respect to $Y$}, which is a map $\Phi^\nu_t : Y \rightarrow X$ by $\Phi^\nu_t(q) := \exp_q(t\nu(q))$, where $\exp_q$ denotes the exponential map at $q$. We can define the map $\Psi : \Sigma \times \mathbb{R} \rightarrow X$ by $\Psi(q,t) := \Phi^\nu_t(q)$. Since $\nu$ and the geodesic flow map are $C^1$-differentiable, so is $\Psi$. Furthermore, $\Psi$ is invertible, as we can send $p \in X$ to the pair $(\pi_{\Sigma}(p), d(p,\Sigma))$, where $\pi_{\Sigma}$ denotes the closest-point projection from $X$ to $\Sigma$. The closest point projection is well-defined, as $\Sigma$ is closed and convex and $X$ is a Hadamard manifold. Since the derivative of $\Psi$ is non-degenerate everywhere, we conclude that $\Psi$ is a $C^1$-diffeomorphism. This allows us to introduce coordinates of the form $(q,t) \in \Sigma \times \mathbb{R}$ on $X$. We will frequently interpret points in $X$ under these coordinates and write $p = (q,t)$.

\subsection{Warped product decompositions of $\Hy^n$}
Let $(M,g)$ and $(N,h)$ be  Riemannian manifolds and let $f\colon N\rightarrow (0,\infty)$ be a smooth positive function. Consider the product $M\times N$ with projections $\pi\colon M\times N\rightarrow M$ and $\eta\colon M\times N \rightarrow N$. 

\begin{definition}The \emph{warped product}  $W=M\times_f N$ is $M\times N$ equipped with the Riemannian metric at $T_{(x,y)}(M\times N)$ given by  \[f(y)\pi^* g+\eta^*h.\]
\end{definition}
Let $\Hy^n$ denote the upper half space model $n$-dimensional real hyperbolic space. That is,
\[\Hy^{n}=\{(x_1,\ldots,x_{n-1}, y)\in \R^{n+1}\mid y>0\}.\]
The Riemannian metric on $\Hy^{n}$ is then given by  \[g_{\Hy^n}=\frac{dx_1^2+\cdots +dx_{n-1}^2+dy^2}{y^2}\]

We recall the following two decompositions of $\Hy^n$ as a warped product.  
\begin{proposition}\label{prop:WarpedDecomposition}Let $c(t)=\cosh(t)$.
\begin{enumerate}
    \item For $n\geq 1$, $\Hy^{n}$ is isometric to the warped product $\Hy^{n-1}\times_{c^2} \R$.
    \item For $n\geq 2$, $\Hy^{n}$ is isometric to the warped product $\Hy^{n-2}\times_{c^2} \Hy^{2}$
\end{enumerate} 
\end{proposition}
In each case the projection onto the first factor is the nearest point projection onto a totally geodesic subspace.  The verification of the warped product structure is a routine calculation.  

With respect to polar coordinates $(r,\theta)$ on $\R^2$, the hyperbolic metric can itself be  written as a warped product $dr^2+\sinh^2(r)d\theta^2$.  Thus in the description from Proposition \ref{prop:WarpedDecomposition}(2) above, the metric on $\Hy^n$ may be written.
\begin{equation}
    dr^2+\sinh^2(r)d\theta^2+\cosh^2(r)dx^2
\end{equation}
where $dx^2$ indicates the metric on $\Hy^{n-2}$.

\subsection{Basics of geodesic flows}

Let $(M,g)$ be a smooth, closed Riemannian manifold. Since $M$ is closed it is geodesically complete, i.e.\,every geodesic can be extended forward and backward for all time. Let $\pi : TM \rightarrow M$ be the canonical projection. Now consider $TTM$, the tangent bundle of $TM$. The covariant derivative allows us to define a projection $\mathcal{K}\colon TTM\rightarrow TM$ as follows. Given $w \in TTM$, let $\gamma : (-\epsilon, \epsilon) \rightarrow TM$ be a path such that $w = \gamma'(0)$. In local coordinates, we can write $\gamma(t) = (p(t),v(t))$. In particular, $v$ is a vector field along the path $p(t)=\pi(\gamma(t))$ in $M$. Set $\mathcal{K}(w) := \frac{Dv}{dt}(0) \in TM$.

The Sasaki metric on $TTM$ is defined to be
\[ g_{\text{Sasaki}}(w, w') = g( d\pi(w), d\pi(w') ) + g( \mathcal{K}(w), \mathcal{K}(w') ). \]

Using the Sasaki metric, we obtain a decomposition of $TTM$ into geodesic, vertical, and horizontal components as follows. Define the vertical part by $\mathbb{V}_{(p,v)} = \ker( d\pi(p,v) )$. The geodesic vector field, which we denote by $\bf{X}$, is defined by the geodesic flow, i.e.\,${\bf X}(p,v)$ is represented by the path $(p(\cdot), p'(\cdot))$, where $p(\cdot)$ is the geodesic induced by the vector $v$. By definition, $\bf{X} \in \ker(\mathcal{K})$. Finally, the horizontal bundle is defined to be the orthogonal complement of $\mathbb{R}\bf{X}$ in $\ker( \mathcal{K} )$ with respect to $g_{\text{Sasaki}}$. 

Note that the entire splitting $TTM = \mathbb{R}X \oplus \mathbb{H} \oplus \mathbb{V}$ is orthogonal with respect to the Sasaki metric. Furthermore, letting $SM$ denote the unit sphere bundle of $M$, the tangent bundle $TSM$ of the unit sphere bundle sits naturally inside $TTM$.   Thus, $TSM$ inherits a splitting as above. Observe that the orthogonal complement of $TSM$ in $TTM$ lies in $\mathbb{V}$. By an abuse of notation, the restriction of $\mathbb{V}$ to $TSM$ will also be called $\mathbb{V}$. The restriction of the Sasaki metric to $TSM$  then provides an orthogonal splitting $TSM = \mathbb{R}\bf{X} \oplus \mathbb{H} \oplus \mathbb{V}$.\\

For any geodesically complete Riemannian manifold, we can define a map $\Phi_{\bullet} : \mathbb{R} \times SM \rightarrow SM$ which sends a pair $(t, (p,v) )$ to $\frac{d}{ds}\vert_{s = t}\exp_p(sv)$, i.e.\,$\Phi_t(p,v)$ is the derivative of the geodesic starting at $p$ along $v$ at time $t$. Note that for each $t$, $\Phi_t$ is a diffeomorphism of $SM$ and $\Phi_t^{-1} = \Phi_{-t}$. Our main technical tool in proving Theorem \ref{thm:MainNoPi1} will be control over the operator norm of $d \Phi_t$.

A Riemannian manifold is called {\it Anosov}, if there exists a continuous, flow invariant splitting of $TSM$ of the form
\[ TSM = \mathbb{R}{\bf X} \oplus E_s \oplus E_u, \]
and there exist $C, \lambda > 0$ such that
\[ \forall t \geq 0, \forall w \in E_s : \vert d \Phi_t(w) \vert \leq C e^{-t \lambda} \vert w \vert, \]
\[ \forall t \leq 0, \forall w \in E_u : \vert d \Phi_t(w) \vert \leq C^{- \vert t \vert \lambda} \vert w \vert, \]
where the metric on $TSM$ inducing $\vert \cdot \vert$ can be any metric. (We will care about the Sasaki metric.) Furthermore, one can show that $\mathbb{H} \oplus \mathbb{V} = E_s \oplus E_u$, which implies that $\mathbb{R} \bf{X}$ is orthogonal to $E_s \oplus E_u$. By Anosov, negatively curved manifolds are Anosov \cite{Anosov67}. It is a well-known fact that for negatively curved manifolds, $\mathbb{H}$ and $\mathbb{V}$ intersect $E_s$ and $E_u$ trivially.

\subsection{Rigidity of hyperbolic spaces}
Let $G$ be a finitely generated torsion-free group. In order to conclude that a given group $G$ is isomorphic to cocompact lattice in $\Isom(\Hy^n)$, we will produce a quasi-isometry from the Cayley graph $G$ with respect to some finite generating set to $\Hy^n$.  The following theorem states that under these conditions, $G$ is isomorphic to a lattice 
\begin{theorem}[\cite{Sullivan81} $n=3$, \cite{Tukia86} $n\geq 4$]\label{thm:HyperRigid} Let $G$ be a torsion-free group that is quasi-isometric to $\Hy^n$, $n\geq 3$. Then $G$ admits a discrete, cocompact action on $\Hy^n$ by isometries.

\end{theorem}

As a corollary, we have

\begin{corollary}\label{cor:QIequalsHE}
Let $M$ be an aspherical closed $n$-manifold, $n\geq 3$.  Then $M$ is homotopy equivalent to a compact quotient of $\Hy^n$ by a discrete torsion-free subgroup if and only if the universal cover $\widetilde{M}$ is quasi-isometric to $\Hy^n$.
\end{corollary}
\begin{proof}
Let $G=\pi_1(M)$. Since $M$ is aspherical and finite dimensional, $G$ must be torsion-free. Suppose $M$ is homotopy equivalent to a manifold $N=\faktor{\Hy^n}{\Gamma}$, where $\Gamma\leq \Isom(\Hy^n)$ is a discrete, torsion-free subgroup.  Then a homotopy equivalence $h\colon M\rightarrow N$ lifts to a proper homotopy equivalence $\widetilde{h}\colon \widetilde{M}\rightarrow \Hy^n$, which is a quasi-isometry.  

Conversely, since $G$ is quasi-isometric to $\til{M}$,  if $\til{M}$ is quasi-isometric to $\Hy^n$ then Theorem \ref{thm:HyperRigid} implies that $G$ is isomorphic to a lattice $\Gamma\leq \Isom(\Hy^n)$. Since $G$ is torsion-free, the quotient $N=\faktor{\Hy^n}{\Gamma}$ is a K(G,1).  In particular, $N$ is homotopy equivalent to $M$.  
\end{proof}

\section{Estimates from curvature bounds}\label{sec:Normal Growth Estimates}

Let $X$ be a Hadamard manifold with pinched negative curvature in $[-b^2, -1]$ and $\Sigma \subset X$ a geodesically convex, codimension-1 $C^2$-differentiable submanifold. Let $\nu$ be a $C^1$-differentiable unit normal vector field on $\Sigma$ (there are exactly two choices for $\nu$). Recall that for $t \in \mathbb{R}$ and $q \in \Sigma$, we defined $\Phi_t^\nu(q) := \exp_q(t\nu(q))$ i.e.\,the flow of $q$ in the direction of $\nu$ for time $t$. For all $t \in \mathbb{R}$, we define $\Sigma_t = \Phi_t^\nu(\Sigma)$. Note that $\Sigma_t \cup \Sigma_{-t}$ is the set of points in $X$ that have distance $\vert t \vert$ to $\Sigma$. Furthermore, since $\mathbb{H}$ is invariant under the geodesic flow, we have that the tangent bundle of $\Phi_t(\nu(\Sigma))$ satisfies $T \Phi_t(\nu(\Sigma)) = \mathbb{H}$ because $T \nu(\Sigma) = \mathbb{H}$ due to geodesic completeness of $\Sigma$ (see section \ref{subsec:TotallyGeodesicSubmanifolds}).

Since $\nu$ is $C^1$-differentiable, so is $\Phi_t^\nu : \Sigma \rightarrow \Sigma_t$. We can thus consider the derivative $D \Phi_t^\nu(q)$ at every point $q \in \Sigma$. Since $\Sigma, \Sigma_t$ are both submanifolds of $X$, they inherit a Riemannian metric, making $d \Phi_t^\nu(q)$ a bounded operator between normed vector spaces. We denote
\[ \Vert d \Phi_t^\nu \Vert_{\infty} := \sup_{q \in \Sigma} \{ \Vert d \Phi_t^\nu(q) \Vert_{op} \}, \]
where $\Vert \cdot \Vert_{op}$ denotes the operator norm. If $\Sigma$ admits a cocompact action by isomtries of $X$ (as it does in the situations we want to consider), then continuity of $d \Phi_t^\nu(q)$ in $q$ implies that $\Vert d \Phi_t^\nu \Vert_{\infty}$ is finite.

\begin{definition}
Let $X, \Sigma, \Phi^\nu$ be as above. We say that {\it the normal growth of $\Sigma$ in $X$ is at most $\beta$} if both limits
\[ \lim_{t \rightarrow \pm \infty} \frac{ \Vert d \Phi_t^\nu \Vert_{\infty} }{ e^{\beta \vert t \vert}} < \infty. \]

\end{definition}

\begin{rmk}
One may wish to define the normal growth exponent to be the infimum of all $\beta$ that satisfy the inequalities above. However, it is a priori not clear whether the infimum still satisfies these inequalities. Since we will need the inequalities to hold, we do not define the normal growth exponent as a number.
\end{rmk}

If $X$ has a negative lower curvature bound, this will provide us with an upper bound on the normal growth of $\Sigma$. This is the content of the next Lemma.

\begin{lemma} \label{lem:CurvatureBoundyieldsNormalGrowthBound}
Suppose there exists a group $H < \Isom(X)$ which leaves $\Sigma$ invariant and acts cocompactly on $\Sigma$. Then there exists a constant $C$ such that for all $p \in \Sigma$, $w \in \mathbb{H}_{\nu(p)}$, $t \in \mathbb{R}$, we have $\vert d \Phi_t(\nu(p))(w) \vert \leq e^{b \vert t \vert} C \vert w \vert$. Furthermore, $\Vert d \Phi_t^\nu \Vert_{\infty} \leq C' e^{b \vert t \vert}$ for all $t \in \mathbb{R}$ and some, possibly different, constant $C'$.
\end{lemma}

\begin{proof}
By Knieper (Lemma 2.16 in \cite{Knieper02}), this splitting satisfies
\[ \forall (p,v) \in SX, \forall w \in E_s, \forall t \in [0, \infty) : \vert d \Phi_t(p,v)(w) \vert \leq e^{-t} \vert w \vert, \]
\[ \forall (p,v) \in SX, \forall w \in E_u, \forall t \in [0, \infty) : \vert d \Phi_t(p,v)(w) \vert \leq e^{bt} \vert w \vert, \]
\[ \forall (p,v) \in SX, \forall w \in E_s, \forall t \in (-\infty, 0] : \vert d \Phi_t(p,v)(w) \vert \leq e^{bt} \vert w \vert, \]
\[ \forall (p,v) \in SX, \forall w \in E_u, \forall t \in (-\infty, 0] : \vert d \Phi_t(p,v)(w) \vert \leq e^{-t} \vert w \vert. \]

If $p \in \Sigma$, $(p,v) = \nu(p)$, $\gamma$ a path in $\Sigma$ through $p$ and $w \in T_{\nu(p)}SX$ the vector corresponding to the path $\nu(\gamma(s))$, we have a unique decomposition $w = w_s + w_u$ with $w_s \in E_s$ and $w_u \in E_u$. For $t \geq 0$, we estimate
\begin{equation*}
    \begin{split}
        \vert d \Phi_t(\nu(p))(w) \vert & \leq \vert d \Phi_t(\nu(p))( w_s ) \vert + \vert d \Phi_t(\nu(p))( w_u) \vert\\
        & \leq e^{-t} \vert w_s \vert + e^{bt} \vert w_u \vert.
    \end{split}
\end{equation*}

Since $w$ is a horizontal vector and $T_{(p,v)}SX$ is finite-dimensional, we obtain that there exist constants $b_s, b_u > 0$ such that for all unit vectors $w \in \mathbb{H}_{(p,v)}$, we have $\vert w_s \vert \leq e_s$ and $\vert w_u \vert \leq e_u$. In particular, for all $t \geq 0$
\begin{equation*}
    \begin{split}
        \vert d \Phi_t(\nu(p))(w) \vert & \leq e^{-t} e_s \vert w \vert + e^{bt} e_u \vert w \vert\\
        & \leq e^{bt} \cdot C_b \vert w \vert,
    \end{split}
\end{equation*}
where $C_b = e_u + e_s$. Since the splittings $E_s \oplus E_u$ and $\mathbb{H} \oplus \mathbb{V}$ are continuous, the constants $e_s, e_u$ are continuous in $p$. In particular, since $H$ acts cocompactly on $\Sigma$, there exists a uniform constant $C_b$, such that for all $p \in \Sigma$, $w \in \mathbb{H}_{(\nu(p))}$, and $t \geq 0$, we have $\vert d \Phi_t(\nu(p))(w) \vert \leq e^{bt} C_b \vert w \vert$.

For $t \leq 0$, one can do the same estimate and argumentation with the other two inequalities we cited from Knieper to obtain that $C_b$ can be chosen such that for all $p \in \Sigma$, $w \in \mathbb{H}_{\nu(p)}$, and $t \in \mathbb{R}$, we have $\vert d \Phi_t(\nu(p))(w) \vert \leq e^{b \vert t \vert} C_b \vert w \vert$.

Since $\Sigma$ is a geodesically convex, codimension-1 submanifold, $\mathbb{H}_{\nu(p)}$ is isometric to $T_p \Sigma$ via the derivative of the projection map $\pi: SX \rightarrow X$. Now consider the restriction of $\pi$ to $\Phi_t(\nu(\Sigma)) \rightarrow \Sigma_t$. We claim that the derivative of this restriction of $\pi$ is an isomorphism at each point. Indeed, consider the map $\nu_t : \Sigma_t \rightarrow \Phi_t(\nu(\Sigma))$ that sends $\Phi^{\nu}_t(q) \in \Sigma_t$ to $\Phi_t(\nu(q))$. This map is an inverse of the restriction of $\pi$. We conclude that the derivative of both of them is a fibrewise isomorphism between the tangentbundles of $\Sigma_t$ and $\Phi_t(\nu(\Sigma))$. We can write $d\Phi_t^\nu$ as the composition
\[ d\Phi_t^\nu = d \pi( \Phi_t(\nu(p)) ) \circ d \Phi_t(\nu(p)) \circ (d \pi(\nu(p)))^{-1}. \]
Since $H < \Isom(X)$ preserves $\Sigma$, there exists a finite index subgroup $H_0$ that preserves $\Sigma_t$ for all $t \in \mathbb{R}$. This subgroup acts cocompactly on $\Sigma$ and on $\Sigma_t$ for all $t \in \mathbb{R}$. We conclude that $d \pi( \Phi_t(\nu(p))$ has uniformly bounded operator norm. The bound on $\vert d \Phi_t(\nu(p))(w) \vert$ thus implies that $\Vert d\Phi_t^\nu \Vert_{\infty} \leq C' e^{bt}$ for all $t \in \mathbb{R}$.
\end{proof}

Analogous to the upper bound proven above, there is a lower bound on $\vert d \Phi_t(\nu(p))(w) \vert$ for $w \in \mathbb{H}$ as well.

\begin{lemma} \label{lem:CurvatureBoundyieldsLowerBound}
Suppose there exists a group $H < \Isom(X)$ which leaves $\Sigma$ invariant and acts cocompactly on $\Sigma$. Then there exists a constant $c > 0$ such that for all $p \in \Sigma$, $w \in \mathbb{H}_{\nu(p)}$, $t \in \mathbb{R}$, we have $\vert d \Phi_t(\nu(p))(w) \vert \geq e^{\vert t \vert} c \vert w \vert$. Furthermore, there exists a constant $c' > 0$ such that for all $t \in \mathbb{R}$, $p \in \Sigma$, $v \in T_p \Sigma$, we have $\vert d \Phi^{\nu}_t (v) \vert \geq c' e^{\vert t \vert} \vert v \vert$.
\end{lemma}

\begin{proof}
Most of the proof is analogous to the proof of Lemma \ref{lem:CurvatureBoundyieldsNormalGrowthBound}, but there are some additional arguments required. First, we show that for any given $T_0 > 0$, there exists a constant $c_1$ such that for all $\vert t \vert \leq T_0$, $p \in \Sigma$, $w \in \mathbb{H}_{\nu(p)}$, we have
\[ \vert d\Phi_t(\nu(p))(w) \vert \geq c_1 \vert w \vert. \]
Indeed, since $\Phi_t^{-1} = \Phi_{-t}$, we find that
\[ \vert w \vert \leq \Vert d\Phi_{-t}( \Phi_t(\nu(p)) ) \Vert_{op} \vert d \Phi_t(\nu(p))(w) \vert. \]
Since $\Vert d \Phi_{-t}( q, v ) \Vert_{op}$ varies continuously in $t$ and $(q,v) \in SX$ and the set $\{ \Phi_t( \nu(p) ) \vert p \in \Sigma, \vert t \vert \leq T_0 \}$ admits a cocompact action by isometries in $H$, there exists a uniform upper bound $C \geq \Vert d \Phi_{-t} ( \Phi_t(\nu(p)) ) \Vert_{op}$. We put $c_1 = \frac{1}{C}$. Note that this implies the Lemma whenever we bound $t$ on a compact interval.

Next we note that, since at every point $(q,v) \in SX$, the horizontal bundle $\mathbb{H}_{(q,v)}$ has trivial intersection with both $E_s$ and $E_u$, that for every $p \in \Sigma$ there exists a constant $C_p$ such that for every $w \in \mathbb{H}$, written in its unique decomposition $w = w_s + w_u$ with $w_s \in E_s$, $w_u \in E_u$, we have $\vert w \vert \leq C_p \min( \vert w_s \vert, \vert w_u \vert )$. Since all the involved bundles vary continuously in $(q,v)$ and $H$ acts cocompactly on $\nu(\Sigma)$, there exists a constant $C_2$ such that for all $p \in \Sigma$ and all $w \in \mathbb{H}_{\nu(p)}$ with decomposition $w = w_s + w_u$, we have $\vert w \vert \leq C_2 \min( \vert w_s \vert, \vert w_u \vert)$.

We are now ready to prove the Lemma for large $t$. We do so by adapting the argument from the proof of Lemma \ref{lem:CurvatureBoundyieldsNormalGrowthBound}. We use two other inequalities from Lemma 2.16 in \cite{Knieper02}, which say
\[ \forall (p,v) \in SX, \forall w \in E_u, \forall t \in [0, \infty) : \vert d \Phi_t(p,v)(w) \vert \geq e^{t} \vert w \vert,  \]
\[ \forall (p,v) \in SX, \forall w \in E_s, \forall t \in (-\infty, 0] : \vert d \Phi_t(p,v)(w) \vert \geq e^{ -t } \vert w \vert. \]

Let $t \geq 0$ and consider $w \in \mathbb{H}_{\nu(p)}$, with its unique decomposition $w = w_s + w_u$. We estimate, using one of the inequalities above and one from the previous Lemma,
\begin{equation*}
    \begin{split}
        \vert d \Phi_t(\nu(p))(w) \vert & \geq \vert d \Phi_t(\nu(p))(w_u) \vert - \vert d \Phi_t(\nu(p))(w_s) \vert\\
        & \geq e^{t} \vert w_u \vert - e^{-t} \vert w_s \vert.
    \end{split}
\end{equation*}
We find $T_0$, such that for $t \geq T_0$, this inequality continues as
\begin{equation*}
    \begin{split}
        \vert d \Phi_t(\nu(p))(w) \vert & \geq e^{t} \vert w_u \vert - e^{-t} \vert w_s \vert\\
        & \geq \frac{1}{2} e^{t} \vert w_u \vert\\
        & \geq \frac{1}{2 C_2} e^{t} \vert w \vert. 
    \end{split}
\end{equation*}

For $t \leq 0$, let $w \in \mathbb{H}_{\nu(p)}$ with unique decomposition $w = w_s + w_u$ and estimate
\begin{equation*}
    \begin{split}
        \vert d \Phi_t(\nu(p))(w) \vert & \geq \vert d \Phi_t(\nu(p))(w_s) \vert - \vert d \Phi_t(\nu(p))(w_u) \vert\\
        & \geq e^{\vert t \vert} \vert w_s \vert - e^{- \vert t \vert} \vert w_u \vert\\
        & \geq \frac{1}{2} e^{ \vert t \vert} \vert w_s \vert\\
        & \geq \frac{1}{2 C_2} e^{ \vert t \vert} \vert w \vert,
    \end{split}
\end{equation*}
where we increased $T_0$ if necessary, such that the inequality holds for all $t \leq - T_0$.

Combining all three cases, we obtain that there exists a constant $c$ such that for all $t \in \mathbb{R}$, all $p \in \Sigma$ and all $w \in \mathbb{H}_{\nu(p)}$, we have
\[ \vert d \Phi(\nu(p))(w) \vert \geq c e^{ \vert t \vert} \vert w \vert. \]

As in the previous proof, we can write $\Phi^{\nu}_t = \pi \circ \Phi_t \circ \nu$. Thinking of $\nu : \Sigma \rightarrow T X$ as a differentiable map, we see that $d \nu$ is a fibrewise isomorphism of vector spaces. Furthermore, the restriction of $d \pi$ to $\mathbb{H}$, which is the image of $d (\Phi_t \circ \nu)$, is an isomorphism as well as we have shown in the previous Lemma. Therefore, we find a constant $c' > 0$ such that for all $t \in \mathbb{R}$, $p \in \Sigma$, and $v \in T_p \Sigma$,
\[ \vert d \Phi^{\nu}_t(v) \vert \geq c' e^{\vert t \vert} \vert v \vert. \]
\end{proof}

\section{Bounding distances with the normal growth exponent} \label{sec:BoundingDistanceswithNormalGrowthExponent}

As we have seen, $X$ is diffeomorphic to $\Sigma \times \mathbb{R}$ via the map that sends $(q,t) \mapsto \Phi_t^{\nu}(q)$, where $\Phi^{\nu}$ denotes the geodesic flow induced by a choice of a unit normal vector field $\nu$ on $\Sigma \subset X$. Using these coordinates, let $p = (q,t),~p' = (q',t') \in X$. We have the following estimate.

\begin{lemma} \label{lem:NormalGrowthInequality}
Suppose the normal growth exponent of $\Sigma$ in $X$ is at most $\beta$. Furthermore, suppose $\Sigma$ is bi-Lipschitz to $\mathbb{H}^{n}$,  let $f : \Sigma \rightarrow \mathbb{H}^{n+1}$ be the composition of this bi-Lipschitz map with an isometric embedding of $\mathbb{H}^n$ into $\mathbb{H}^{n+1}$, and let $\Phi_t^{\perp}$ denote the geodesic flow induced by a choice of unit normal vector field on $\mathbb{H}^n \subset \mathbb{H}^{n+1}$. Then there exists a constant $C$ such that for all $p = (q,t)$, $p' = (q',t') \in X$,
\[ d(p, p') \leq C d_{\mathbb{H}^{n+1}}( p_{\beta}, p'_{\beta} ),  \]
where $p_{\beta} = \Phi_{\beta t}^{\perp}(f(q))$ and $p'_{\beta} = \Phi_{\beta t'}^{\perp}(f(q'))$.
\end{lemma}

\begin{proof}
We start by proving this inequality for the case where $t, t' \geq 0$. (The other cases will follow from this.) Let $q, q' \in \Sigma$, $t, t' \geq 0$. Let $\tilde{\gamma}$ be the unique geodesic from $p_{\beta}$ to $p'_{\beta}$. Since the chosen isometric embedding of $\mathbb{H}^n \hookrightarrow \mathbb{H}^{n+1}$ provides $\mathbb{H}^{n+1}$ with a diffeomorphism to $\mathbb{H}^n \times \mathbb{R}$ equipped with a warped product metric, we can write $p_{\beta} = (f(q), \beta t), p'_{\beta} = (f(q'), \beta t')$ analogously as we did with $p, p'$ in $X$. In this notation, we can write $\tilde{\gamma} = (\tilde{\gamma}_{\mathbb{H}^n}, \tilde{\gamma}_{\perp})$, where $\tilde{\gamma}_{\mathbb{H}^n}$ is the unique geodesic from $f(q)$ to $f(q')$.

Let $\epsilon > 0$. Let $\gamma_{\Sigma}$ be a $C^1$-path in $\Sigma$ from $q$ to $q'$ such that $f \circ \gamma_{\Sigma}$ satisfies $l( f \circ \gamma_{\Sigma} ) \leq d_{\mathbb{H}^n} ( f(q), f(q') ) + \epsilon$. Such a path can be obtained by approximating the preimage of the geodesic from $f(q)$ to $f(q')$ under $f^{-1}$ suitably well. We obtain a $C^1$-differentiable path $\gamma = (\gamma_{\Sigma}, \frac{1}{\beta} \tilde{\gamma}_{\perp} )$ from $(q, t)$ to $(q', t')$. We consider the orthogonal projections of $\gamma'(s)$ to $\nu(\gamma(s))$ and to $\nu(\gamma(s))^{\perp}$, which we denote by $\gamma'_{radial}(s)$ and $\gamma'_{slice}(s)$ respectively. We have the following identities:
\[ \Vert \gamma'_{radial}(s) \Vert = \frac{1}{\beta} \Vert \tilde{\gamma}'_{\perp}(s) \Vert, \]
\[ \gamma'_{slice}(s) = d\Phi^\nu_{\frac{1}{\beta} \tilde{\gamma}_{\perp}(s)}(\gamma'_{\Sigma}(s)). \]

These identities allow us to estimate the length of $\gamma$. Namely, using the fact that the normal growth exponent of $X$ with respect to $\Sigma$ is at most $\beta$, we obtain
\begin{equation*}
    \begin{split}
    l(\gamma) & = \int_{0}^1 \sqrt{ \langle \gamma'_{slice}(s), \gamma'_{slice}(s) \rangle + \langle \tilde{\gamma}'_{radial}(s), \tilde{\gamma}'_{radial}(s) \rangle } ds\\
    & = \int_{0}^1 \sqrt{ \langle d\Phi^\nu_{\frac{1}{\beta} \tilde{\gamma}_{\perp}(s)}( \gamma'_{\Sigma}(s) ), d\Phi^\nu_{\frac{1}{\beta} \tilde{\gamma}_{\perp}(s)}( \gamma'_{\Sigma}(s) ) \rangle + \frac{1}{\beta^2} \langle \tilde{\gamma}'_{\perp}(s), \tilde{\gamma}_{\perp}(s) \rangle } ds\\
    & \leq C \int_{0}^1 \sqrt{ e^{2 \beta \frac{1}{\beta} \tilde{\gamma}_{\perp}(s)} \langle \gamma'_{\Sigma}(s), \gamma'_{\Sigma}(s) \rangle + \frac{1}{\beta^2} \langle \tilde{\gamma}'_{\perp}(s), \tilde{\gamma}'_{\perp}(s) \rangle } ds\\
    & \leq \sqrt{4}C \int_{0}^1 \sqrt{ \cosh(\tilde{\gamma}_{\perp}(s))^2 \langle \gamma'_{\Sigma}(s), \gamma'_{\Sigma}(s) \rangle + \frac{1}{\beta^2} \langle \tilde{\gamma}'_{\perp}(s), \tilde{\gamma}'_{\perp}(s) \rangle } ds.
    \end{split}
\end{equation*}

Let $L$ be the bi-Lipschitz constant of $f$. Since $f$ is bi-Lipschitz, we have that $f \circ \gamma_{\Sigma}$ is differentiable almost everywhere. Wherever it is, the inequality $\Vert \gamma'_{\Sigma}(s) \Vert \leq L \Vert (f \circ \gamma_{\Sigma})'(s) \Vert$ holds. In particular, we can continue the estimate above as
\begin{equation*}
    \begin{split}
        l(\gamma) & \leq 2C \int_{0}^1 \sqrt{ \cosh(\tilde{\gamma}_{\perp}(s))^2 \langle \gamma'_{\Sigma}(s), \gamma'_{\Sigma}(s) \rangle + \frac{1}{\beta^2} \langle \tilde{\gamma}'_{\perp}(s), \tilde{\gamma}'_{\perp}(s) \rangle } ds\\
        & \leq 2LC \int_{0}^1 \sqrt{ \cosh(\tilde{\gamma}_{\perp}(s))^2 \langle (f \circ \gamma_{\Sigma})'(s), (f \circ \gamma_{\Sigma})'(s) \rangle + \frac{1}{\beta^2} \langle \tilde{\gamma}'_{\perp}(s), \tilde{\gamma}'_{\perp}(s) \rangle } ds\\
        & \leq 2LC \int_{0}^1 \sqrt{ \cosh(\tilde{\gamma}_{\perp}(s))^2 \langle (f \circ \gamma_{\Sigma})'(s), (f \circ \gamma_{\Sigma})'(s) \rangle +  \langle \tilde{\gamma}'_{\perp}(s), \tilde{\gamma}'_{\perp}(s) \rangle } ds.
    \end{split}
\end{equation*}

The integral in the last expression is the length of the path $(f \circ \gamma_{\Sigma}, \tilde{\gamma}_{\perp})$ with respect to the warped product metric on $\mathbb{H}^n \times \mathbb{R}$, which is isometric to $\mathbb{H}^{n+1}$. We obtain
\begin{equation*}
    \begin{split}
        l(\gamma) \leq 2LC l_{\mathbb{H}^{n+1}}( (f \circ \gamma_{\Sigma}, \tilde{\gamma}_{\perp}) ),
    \end{split}
\end{equation*}
which is a path from $p_{\beta}$ to $p'_{\beta}$.

Let $\delta > 0$. Since geodesics in $\mathbb{H}^{n+1}$ project to (unparametrised) geodesics in the isometrically embedded $\mathbb{H}^n$, there exists $\epsilon > 0$, depending on $p$, $p'$, and $\delta$, such that if $l_{\mathbb{H}^n}(f \circ \gamma_{\Sigma}) \leq d_{\mathbb{H}^n}(f(q), f(q')) + \epsilon$, then $l_{\mathbb{H}^{n+1}}( (f \circ \gamma_{\Sigma}, \tilde{\gamma}_{\perp}) ) \leq d_{\mathbb{H}^{n+1}}(p_{\beta}, p'_{\beta}) + \delta$. Starting with some given $\delta$, we can thus choose $\gamma_{\Sigma}$ such that
\[ d(p, p') \leq l(\gamma) \leq 2LC l_{\mathbb{H}^{n+1}}( (f \circ \gamma_{\Sigma}, \tilde{\gamma}_{\perp}) ) \leq 2LC ( d_{\mathbb{H}^{n+1}}( p_{\beta}, p'_{\beta} ) + \delta ). \]
As $\delta$ tends to zero, we obtain, $d(p, p') \leq 2LC d_{\mathbb{H}^{n+1}}(p_{\beta}, p'_{\beta})$. This proves the inequality in the case $t, t' \geq 0$.\\

If $t, t' \leq 0$, the same estimates go through. If $t < 0 < t'$ and $\tilde{\gamma}$ is the geodesic from $(f(q), \beta t)$ to $(f(q'), \beta t')$, there exists a unique point $p_0 = (q_0, 0) \in \Sigma \subset X$ such that $\tilde{\gamma}$ intersects $f(\Sigma)$ in $(f(q_0), 0)$. We estimate
\begin{equation*}
    \begin{split}
        d( (q,t), (q',t') ) & \leq d( (q,t) , (q_0, 0) ) + d( (q_0, 0), (q', t') )\\
        & \leq d( (f(q), \beta t), (f(q_0), 0) ) + d( (f(q_0), 0), (f(q'), \beta t') )\\
        & = d( (f(q), \beta t) , (f(q'), \beta t') ).
    \end{split}
\end{equation*}
This proves the Lemma.
\end{proof}

\begin{lemma} \label{lem:LowerBound}
Let $X, \Sigma, f$, and $\Phi^{\perp}_t$ be as in Lemma \ref{lem:NormalGrowthInequality}, except that we do not assume any bound on the normal growth exponent. Suppose instead that there exists $H < \Isom(X)$ that preserves $\Sigma$ and acts cocompactly on $\Sigma$. Then there exists a constant $c$ such that for all $p = (q,t)$, $p' = (q',t') \in X$,
\[ d(p,p') \geq c d_{\mathbb{H}^{n+1}} ( \overline{p}, \overline{p'} ), \]
where $\overline{p} = \Phi_t^{\perp}(f(q))$ and $\overline{p'} = \Phi_{t'}^{\perp}(f(q'))$.
\end{lemma}

\begin{proof}
The proof follows the same argument as the proof of Lemma \ref{lem:NormalGrowthInequality}, except that we use the lower bound on $\vert d \Phi_t(\nu(p))(w) \vert$ obtained in Lemma \ref{lem:CurvatureBoundyieldsLowerBound} instead of the bound given by the normal growth exponent.

Let $t, t' \in \mathbb{R}$, $q, q' \in \Sigma$. Let $\gamma$ be the unique geodesic from $p$ to $p'$. Using the diffeomorphism $X \approx \Sigma \times \mathbb{R}$, we can write $\gamma = (\gamma_{\Sigma}, \gamma_{\perp})$. As before, we have that $\gamma'(s)$ can be decomposed into its orthogonal projections to $\nu(\gamma(s))$ and $\nu(\gamma(s))^{\perp}$, which we denote by $\gamma'_{radial}(s)$ and $\gamma'_{slice}(s)$ respectively and we have the identities
\[ \Vert \gamma'_{radial}(s) \Vert = \Vert \gamma'_{\perp}(s) \Vert, \]
\[ \gamma'_{slice}(s) = d \Phi_{\gamma_{\perp}}^\nu( \gamma'_{\Sigma}(s) ). \]

Using Lemma \ref{lem:CurvatureBoundyieldsLowerBound} together with the fact that $\gamma'_{\Sigma}(s) \in T \Sigma$ for all $s$, we estimate
\begin{equation*}
    \begin{split}
        l(\gamma) & = \int_{0}^1 \sqrt{ \langle \gamma'_{slice}(s), \gamma'_{slice}(s) \rangle + \langle \gamma'_{radial}(s), \gamma'_{radial}(s) \rangle } ds\\
        & = \int_{0}^1 \sqrt{ \langle d \Phi^\nu_{\gamma_{\perp}(s)}( \gamma'_{\Sigma}(s) ), d \Phi^\nu_{\gamma_{\perp}(s)}( \gamma'_{\Sigma}(s) ) \rangle + \langle \gamma'_{\perp}(s), \gamma'_{\perp}(s) \rangle } ds\\
        & \geq c' \int_{0}^1 \sqrt{ e^{ 2 \vert \gamma_{\perp}(s) \vert } \langle \gamma'_{\Sigma}(s), \gamma'_{\Sigma}(s) \rangle + \langle \gamma'_{\perp}(s), \gamma'_{\perp}(s) \rangle } ds\\
        & \geq c' \int_{0}^1 \sqrt{ \cosh( \gamma_{\perp}(s) )^2 \langle \gamma'_{\Sigma}(s), \gamma'_{\Sigma}(s) \rangle + \langle \gamma'_{\perp}(s), \gamma'_{\perp}(s) \rangle } ds\\
        & \geq \frac{c'}{L} \int_{0}^1 \sqrt{ \cosh(\gamma_{\perp}(s) )^2 \langle (f \circ \gamma_{\Sigma})'(s), (f \circ \gamma_{\Sigma})'(s) \rangle + \langle \gamma'_{\perp}(s), \gamma'_{\perp}(s) \rangle } ds\\
        & = \frac{c'}{L} l_{\mathbb{H}^{n+1}}( f \circ \gamma_{\Sigma}, \gamma_{\perp} ),
    \end{split}
\end{equation*}
where $( f \circ \gamma_{\Sigma}, \gamma_{\perp})$ is a continuous path in $\mathbb{H}^{n+1}$ from $(f(q),t)$ to $(f(q'),t')$ that is differentiable almost everywhere. We conclude that
\[ l(\gamma) \geq \sqrt{ \frac{c'}{L} } l_{\mathbb{H}^{n+1}} ( f \circ \gamma_{\Sigma}, \gamma_{\perp} ) \geq \sqrt{ \frac{c'}{L} } d_{\mathbb{H}^{n+1}}( (f(q),t), (f(q'),t') ), \]
which is the inequality we wanted to prove.
\end{proof}

As a corollary, we obtain Theorem \ref{thm:MainNoPi1}.

\begin{corollary} \label{cor:BiLipschitz}
Let $X$ be a Hadamard manifold with sectional curvatures in $[-b^2, -1]$ and $\Sigma$ a geodesically convex, $C^2$-submanifold such that the normal growth exponent of $\Sigma$ in $X$ is at most $1$. Suppose there exists $H < \Isom(X)$ that preserves $\Sigma$ and acts cocompactly on it and suppose $\Sigma$ is bi-Lipschitz equivalent to $\mathbb{H}^n$. Then $X$ is bi-Lipschitz equivalent to $\mathbb{H}^{n+1}$.
\end{corollary}

\begin{proof}
Since $\Sigma$ is geodesically convex, $X$ is $C^2$-diffeomorphic to $\Sigma \times \mathbb{R}$. Similarly, we write $\mathbb{H}^{n+1}$ as a warped product $\mathbb{H}^n \times_{\cosh(t)^2} \mathbb{R}$. Let $f : \Sigma \rightarrow \mathbb{H}^{n+1}$ be the composition of the bi-Lipschitz map $\Sigma \rightarrow \mathbb{H}^n$ with the isometric identification of $\mathbb{H}^n$ with $\mathbb{H}^n \times \{ 0 \} \subset \mathbb{H}^n \times_{\cosh(t)^2} \mathbb{R}$. Define $F : X \mapsto \mathbb{H}^{n+1}$ by $F((q,t)) := (f(q), t)$. Applying Lemma \ref{lem:NormalGrowthInequality} for the case $\beta = 1$ and Lemma \ref{lem:LowerBound} we find a constant $\overline{L}$ such that
\[ \frac{1}{\overline{L}} d(F(p), F(p')) \leq d(p,p') \leq \overline{L} d(F(p), F(p') ). \]
In other words, $F$ is a bi-Lipschitz map between $X$ and $\mathbb{H}^{n+1}$ (bijectivity is clear).
\end{proof}

\begin{corollary}\label{cor:Growth Exponent Bound implies lattice} Let $X$ be a Hadamard manifold with sectional curvatures in $[-b^2, -1]$ and $\Sigma$ a geodesically convex, $C^2$-submanifold such that the normal growth exponent of $\Sigma$ in $X$ is at most $1$.  Suppose $G$ acts cocompactly on $X$ and that the stabilizer $H$ of $\Sigma$ acts cocompactly on $\Sigma$.  If $\Sigma$ is bi-Lipschitz to $\Hy^n$ then $G$ is isomorphic to a cocompact lattice in $\Isom(\Hy^{n+1})$.
\end{corollary}
\begin{proof}
If $n+1= 2$, then $G$ must be the fundamental group of a surface of genus $g\geq 2$, hence $G$ is obviously a cocompact lattice in $\Isom(\Hy^{2})$. We may thus suppose $n+1\geq 3$. Since $G$ acts cocompactly on $X$, the Cayley graph of $G$ is quasi-isometric to $X$, and hence by Corollary \ref{cor:BiLipschitz}, $G$ is quasi-isometric to $\Hy^{n+1}$.  Thus, $G$ acts on $\dinf \Hy^{n+1}$ by uniform quasi-symmetric homeomorphisms.  By Theorem \ref{thm:HyperRigid}, $G$ is isomorphic to a cocompact lattice in $\Isom(\Hy^{n+1})$.
\end{proof}

\begin{proof}[Proof of Theorem \ref{thm:Main Cocompact}]By Corollary \ref{cor:Growth Exponent Bound implies lattice}, $G=\pi_1(M)$ is isomorphic to a cocompact, torsion-free lattice $\Gamma\leq \Isom(\Hy^{n+1})$. Thus, $\faktor{\Hy^{n+1}}{\Gamma}$ is a $K(G,1)$, hence homotopy equivalent to $M$.  
\end{proof}

We highlight that the proof of Corollary \ref{cor:BiLipschitz} does not use Lemma \ref{lem:CurvatureBoundyieldsNormalGrowthBound}. The lower curvature bound can only show that the normal growth exponent is at most $b$, but the corollary requires that the normal growth exponent is at most $1$. In the next section, we present a family of examples that highlight that a lower bound on the curvature cannot suffice to obtain the conclusion of Corollary \ref{cor:BiLipschitz}.

\section{Gromov--Thurston Manifolds} \label{sec:GromovThurstonManifolds}

In this section, we prove Theorem \ref{thm:MainGromovThurstonExamples}. As stated in the introduction, these examples arise from a construction due to Gromov and Thurston.  We begin by reviewing this construction (see \cite{GromovThurston87,Kapovich07} for more details).


  Suppose $n\geq 4$ and let $I_n=\sum_{i=1}^nx_i^2-\sqrt{2}x_{n+1}^2$  and denote by $\Gamma_n$ the group of automorphisms of $I_n$ with entries in the ring of integers of $\Q(\sqrt{2})$. The inner product $I_n$ has signature $(n,1)$ and therefore we can identify $\Hy^{n}$ with $\{x\in \R^{n+1}\mid I_n(x)=-1\}$. From this we see that $\Gamma_n$ acts on $\Hy^n$, and the action is discrete and cocompact.   Moreover, since $\Gamma_n$ is linear, there exists a torsion-free normal subgroup $\Gamma \trianglelefteq \Gamma_n$ of finite-index. The quotient $M=\faktor{\Hy^n}{\Gamma}$ is a closed manifold, and passing to a double cover if necessary, we may assume $M$ is orientable. We will call such an $M$ an $I_n$-manifold.

Since $\Gamma$ is normal, the reflection $ x_1\mapsto -x_1$ descends to an involution $r_1$ on $M$ with fixed set a closed, totally geodesic, codimension-1 submanifold $W_1$. The involution $x_2\mapsto-x_2$ also induces an involution $r_2$ of $M$ which maps $W_1$ to itself and fixes a closed, totally geodesic codimension-1 submanifold $W_2$. Thus, $V=W_1\cap W_2$ is  a totally geodesic codimension-2 submanifold of $M$ and the dihedral angle of $W_1$ and $W_2$ along $V$ is $\pi/2$.

Denote by $U_1$ and $U_2$ the closure of each component of $W_1\setminus V$. In particular, we have $r_2\colon U_1\rightarrow U_2$ is an isometry.  Since $V$ bounds both $U_1$, and $U_2$, the homology class $[V]\in H_{n-2}(M;\Z)$ is trivial. 

Let $\Mtil_k$ be the $k$-fold branched cover of $M$, branched along $V$, and denote by $\psi_k\colon \Mtil_k\rightarrow M$ the covering map. The involutions $r_i,~i=1,2$ lift to $2k$ involutions of $\Mtil_k$, which therefore generate a dihedral group $D_{2k}$. If $F$ is the manifold with corners obtained by splitting $M$ open along $U_2$, then $F$ is a fundamental domain for the action of $\faktor{\Z}{k\Z}$ on $\Mtil_k$ that induces the covering.  There are $k$ copies of $F$ in $\Mtil_k$ which are cyclically permuted.  Label these $F^0,\ldots, F^{k-1}$ as in Figure \ref{fig:fundamental domain}.  There are also $k$ lifts of $U_1$ and $U_2$ in $\Mtil_k$, and we label these $\what{U}_1^i$ and $\what{U}_2^i$ respectively, where $0\leq i\leq k-1$. Finally, the branch locus $\what{V}\subset \Mtil_k$ is fixed by the action of $\faktor{\Z}{k\Z}$, and maps isometrically down to $V$.  

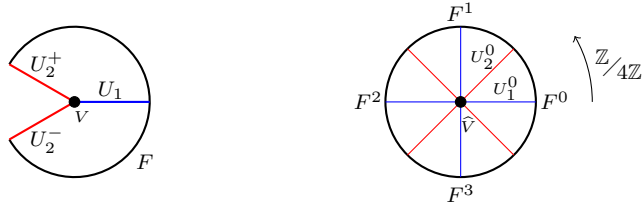
\begin{figure}[h]
 
  \begin{tikzpicture}
\draw[thick] (0,0) arc (-150:150:1);

\draw[thick,red] (0,0)--(.866,.5);
\draw[thick, red] (0,1)--(.866,.5);
\draw[thick, blue](.866,.5)--(1.866,.5);
\node at (1.8,-.3) {$\scriptstyle F$};
\node at (1.35,.65) {$\scriptstyle U_1$};
\node at (.5,1) {$\scriptstyle U_2^+$};
\node at (.5,0) {$\scriptstyle U_2^-$};
\node at (.966,.3) {$\scriptscriptstyle V$};
\filldraw (.866,.5) circle (2pt);

\draw[thick] (6,.5) circle (1); 
\draw[blue] (6,-.5)--(6,1.5);
\draw[blue] (5,.5)--(7,.5);
\draw[red] (5.293,-.207)--(6.707,1.207);
\draw[red] (5.293,1.207)--(6.707,-.207);
\filldraw (6,.5) circle (2pt);
\node at (7.25,.52){$\scriptstyle F^0$};
\node at (4.77,.52){$\scriptstyle F^2$};
\node at (6,1.7){$\scriptstyle F^1$};
\node at (6,-.72){$\scriptstyle F^3$};
\node at (6.6, .68){$\scriptscriptstyle U_1^0$};
\node at (6.3,1.1){$\scriptscriptstyle U_2^0$};
\node at (6.1,.2) {$\scriptscriptstyle \what{V}$};
\draw[ ->] (7.75,.5) arc (0:30:1.75);
\node at (8.1,1){$\scriptstyle \faktor{\Z}{4\Z}$};
\end{tikzpicture}
    \caption{A schematic for the fundamental domain $F$ is shown on the left and on the right, a schematic for $\Mtil_4$. The central black circles represent $V$ and $\what{V}$, respectively. The dihedral angle between adjacent red and blue rays is $\pi/2$.}
    \label{fig:fundamental domain}
\end{figure}

At a point $p\in V$, the metric $g$ on $M$ can be written as:
\[dr^2+\sinh^2(r)d\theta^2+\cosh^2(r)dx^2\]
where $(r,\theta)$ are coordinates on a disk orthogonal to $V$ through $p$, and $dx^2$ is the metric on $\Hy^{n-2}$. These coordinates are valid for $0\leq \theta\leq 2\pi$ and $0\leq r< \rho$, the normal injectivity radius of $V$. The pullback $h_k=\psi_k^*(g)$ is a singular Riemannian metric  on $\Mtil_k$, but away from $\what{V}$, $h_k$ is locally isometric to $\Hy^n$.  Along $\what{V}$, for $0\leq r< \rho$, the circumference of a orthogonal transverse disk of radius $r$ is $2\pi k\sinh(r)$. 

\begin{lemma}[\cite{GromovThurston87}, Remark 3.6]\label{lem:Not Homotopy Equivalent}
For each $k>1$, $\Mtil_k$ is not homotopy equivalent to a manifold of constant curvature $\kappa\equiv -1$.
\end{lemma}
\begin{rmk}
Remark 3.6 of \cite{GromovThurston87} only states that $\Mtil_k$ does not admit a metric with $\kappa\equiv-1$. However, the existence of the $D_{2k}$-action on $\Mtil_k$ together with Mostow rigidity allows the same argument to go through just assuming $\Mtil_k$ is homotopy equivalent to a manifold of constant negative curvature.

\end{rmk}

In \cite{GromovThurston87}, it is shown that one can then smooth  $h_k$ to obtain a metric $g_k$ which agrees with $h_k$ outside of the $\rho$-neighborhood of $V$, and which can be described in a $\rho$ neighborhood of $V$ by \[dr^2+\sigma^2(r)d\theta^2+\cosh^2(r)dx^2\] where $\sigma(r)$ is a function that smoothly interpolates between $\sinh(r)$ and $k\sinh(r)$ on the interval $[r_0,\rho]$ for some $0<r_0<\rho$. The function $\sigma$ also satisfies other properties \cite[Lemma 2.1]{GromovThurston87} to ensure that $g_k$ is still negatively curved:

\begin{lemma}[\cite{GromovThurston87}, Lemma 2.4]\label{lem:curvature of Mk}
Let $\rho$ be the normal injectivity radius of $V\subset M$.  There exists a constant $C(k,\rho)>1$ such that the curvature of $(\Mtil_k,g_k)$ satisfies
\[-C(k,\rho)\leq\kappa(g_k)\leq\frac{-1}{C(k,\rho)}\]
Moreover,  $C(k,\rho)\rightarrow 1$ as $\rho\rightarrow \infty$ for fixed $k$.
\end{lemma}

By choosing $\Gamma$ to be a sufficiently deep finite index subgroup of $\Gamma_n$, one can ensure that $\rho$ is arbitrarily large. After rescaling the metric $g_k$, Lemma \ref{lem:curvature of Mk} now implies:
\begin{lemma}[\cite{GromovThurston87}]\label{lem:epsilon curvature}For every $\epsilon>0$ and every $k>1$ there exists an $I_n$-manifold $M$ and a branched cover $(\Mtil_k,g_k)$ of $M$ constructed as above such that \[-1-\epsilon\leq \kappa(g_k)\leq -1\]
\end{lemma}

We next construct a totally geodesic, codimension-1 submanifold in $(\Mtil_k,g_k)$ that is locally isometric to $\Hy^{n-1}$.  Let $r_1^0$ be the lift of $r_1$ to $\Mtil_k$ that sends $F^0$ to itself and fixes $\what{U}_1^0$.  When $k$ is even, $r_1^0$ also fixes $\what{U}_1^{\frac{k}{2}}$, and sends $F^{\frac{k}{2}}$ to itself. When $k$ is odd, $r_0^1$ sends $F^{\frac{k-1}{2}}$ to $F^{\frac{k+1}{2}}$, and fixes $U_2^{\frac{k-1}{2}}$. Define
\[\what{W}=\left\{\begin{array}{cl}
\what{U}_1^0\cup \what{U}_2^{\frac{k-1}{2}}, & k~\text{odd}\\
\what{U}_1^0\cup \what{U}_1^{\frac{k}{2}}, &k~\text{even}
\end{array}\right.
\]

\begin{proposition}\label{prop:W is totally geodesic}
The hypersurface $\what{W}$ is totally geodesic in $(\Mtil_k,g_k)$ and locally isometric to $\Hy^{n-1}$.
\end{proposition}
\begin{proof}
Since $g_k$ is rotationally symmetric near $\what{V}$, the involution $r_1^0$ is also an isometry of $(\Mtil_k,g_k)$. In particular, $\what{W}$ must be totally geodesic by Theorem 1.10.15 of \cite{Klingenberg95}. 

For the second part, observe that for constant $\theta$ the restriction of $g_k$ to $\what{W}$ agrees with the restriction $h_k$ to $\what{W}$, \emph{i.e.} locally isometric to $\Hy^{n-1}$.
\end{proof}

\begin{proof}[Proof of Theorem \ref{thm:MainGromovThurstonExamples}] 
Fix any $k\geq2$ and let $\epsilon_i\to 0$ be any sequence of positive real numbers. By Lemma \ref{lem:epsilon curvature}, for each $\epsilon_i$ we find an $I_n$-manifold $M_i'$ such that the $k$-fold branched cover $\Mtil_{i,k}$ of $M_i'$ admits a metric $g_{i,k}$ whose curvature satisfies $-1-\epsilon_i\leq \kappa(g_{i,k})\leq-1$. By Lemma \ref{lem:Not Homotopy Equivalent}, no $\Mtil_{i,k}$ is homotopy equivalent to a manifold of constant curvature $\kappa\equiv-1$, hence by Corollary \ref{cor:QIequalsHE}, we know that $\pi_1(\Mtil_{i,k})$ is not quasi-isometric to $\Hy^n$. 

By Proposition \ref{prop:W is totally geodesic}, each $\Mtil_{i,k}$ contains a totally geodesic codimension-1 submanifold $\what{W}_{i,k}$ that is locally isometric to $\Hy^{n-1}$.  It follows from the curvature bound and Lemma \ref{lem:CurvatureBoundyieldsNormalGrowthBound} that the normal growth exponent of $\what{W}_{i,k}$ in $\Mtil_{i,k}$ is at most $1+\epsilon_i$. Setting $M_i=\Mtil_{i,k}$ and $N_i=\what{W}_{i,k}$ as in the statement of the theorem completes the proof.
\end{proof}

\bibliography{Codim1bib}
\bibliographystyle{alpha}

\end{document}